\documentclass[a4paper,12pt]{amsart}

\usepackage{amsmath}
\usepackage{amssymb}
\usepackage{mathrsfs}
\usepackage{enumerate}
\usepackage{ifthen}
\usepackage{graphicx}
\usepackage[T1]{fontenc} 

\nonstopmode \numberwithin{equation}{section}
\newtheorem{thm}{Theorem}[section]
\newtheorem{cor}{Corollary}[section]
\newtheorem{lem}{Lemma}[section]

\theoremstyle{definition}

\newcounter{minutes}
\divide\time by 60
\newcounter{hours}
\multiply\time by 60
\addtocounter{minutes}{-\time}

\newcounter {own}
\def\theown {\thesection       .\arabic{own}}

{\qed\bigskip}

\newcounter{alphabet}

\begin{document}

\title{PRE-SCHWARZIAN NORM ESTIMATES FOR THE CLASS OF JANOWSKI STARLIKE FUNCTIONS}

\author{Md Firoz Ali}
\address{Md Firoz Ali,
Department of Mathematics, National Institute of Technology Durgapur,
Durgapur- 713209, West Bengal, India.}
\email{ali.firoz89@gmail.com, firoz.ali@maths.nitdgp.ac.in}

\author{Sanjit Pal}
\address{Sanjit Pal,
Department of Mathematics, National Institute of Technology Durgapur,
Durgapur- 713209, West Bengal, India.}
\email{palsanjit6@gmail.com}


\def\thefootnote{}
\footnotetext{ {\tiny File:~\jobname.tex,
printed: \number\year-\number\month-\number\day,
          \thehours.\ifnum\theminutes<10{0}\fi\theminutes }
} \makeatletter\def\thefootnote{\@arabic\c@footnote}\makeatother

\begin{abstract}
For $-1\leq B<A\leq 1$, let $\mathcal{S}\mbox{*}(A,B)$ denote the class of Janowski starlike functions which satisfy the subordination relation $zf'(z)/f(z)\prec (1+Az)/(1+Bz)$. In the present article, we determine the sharp of pre-Schwarzian norm for the functions in the class $\mathcal{S}\mbox{*}(A,B)$.
\end{abstract}

\thanks{}

\maketitle
\pagestyle{myheadings}
\markboth{Md Firoz Ali and Sanjit Pal}{The pre-Schwarzian norm estimates for the class of Janowski starlike functions}

\section{Introduction}

Let $\mathcal{A}$ denote the class of analytic functions in the open unit disc $\mathbb{D}=\{z\in\mathbb{C}:|z|<1\}$ with the Taylor series expansion
\begin{equation}\label{p-00001}
f(z)= z+\sum_{n=2}^{\infty}a_n z^n.
\end{equation}
Let $\mathcal{S}$ denote the subclass of $\mathcal{A}$ consisting of functions which are univalent. A function $f\in\mathcal{A}$ is called starlike (respectively convex) if $f$ is univalent and the image $f(\mathbb{D})$ is starlike with respect to the origin (respectively convex). The classes of all starlike and convex functions are denoted by $\mathcal{S^*}$ and $\mathcal{C}$ respectively. It is well known that a function $f\in\mathcal{A}$ is starlike if and only if ${\rm Re\,} [zf'(z)/f(z)] > 0$ for $z\in\mathbb{D}$ and a function $f\in\mathcal{A}$ is convex if and only if ${\rm Re\,} [1 + zf''(z)/f'(z)] > 0$ for $z\in\mathbb{D}$. See \cite{Duren-1983,Goodman-book-1983} for further information on these classes.\\

Let $f$ and $g$ be two analytic functions in $\mathbb{D}$. The function $f$ is said to be subordinate to $g$ if there exists an analytic function $\omega:\mathbb{D}\rightarrow\mathbb{D}$ with $\omega(0)=0$ such that $f(z)=g(\omega(z))$ and it is denoted by $f\prec g$. Moreover, when $g$ is univalent, $f\prec g$ if and only if $f(0)=g(0)$ and $f(\mathbb{D})\subset g(\mathbb{D})$.\\

For $-1\leq B<A\leq 1$, let $\mathcal{S}\mbox{*}(A,B)$ denote the class of functions $f\in\mathcal{A}$ satisfying
\begin{equation}\label{p-00005}
\frac{zf'(z)}{f(z)}\prec \frac{1+Az}{1+Bz}\quad\text{for}~z\in\mathbb{D}.
\end{equation}
The class  $\mathcal{S}\mbox{*}(A,B)$ was first introduced and studied by Janowski \cite{Janowski-1973} in $1973$. It is easy to see that the functions in  $\mathcal{S}\mbox{*}(A,B)$ are also starlike functions. Several mathematicians studied the class  $\mathcal{S}\mbox{*}(A,B)$ with an emphasis on particular values of $A$ and $B$. For particular values of the parameters $A$ and $B$, one can obtain various subclasses studied by several authors. For example, we list down some of the subclasses as:
\begin{enumerate}[(i)]
\item $\mathcal{S}\mbox{*}(1,-1)=:\mathcal{S}\mbox{*}$ is the family of starlike functions.
\item For $0\leq\alpha<1$, $\mathcal{S}\mbox{*}(1-2\alpha,-1)=:\mathcal{S}\mbox{*}(\alpha)$ is the family of starlike functions of order $\alpha$.
\item The family $\mathcal{S}\mbox{*}(1,0)$ was introduced and studied by Singh and Singh \cite{Singh-Singh-1974}.
\item For $0<\alpha\leq 1$, the family $\mathcal{S}\mbox{*}(\alpha,-\alpha)=:\mathcal{S}(\alpha)$ was introduced and studied by Padmanabhan \cite{Padmanabhan-1968}.
\end{enumerate}

\noindent For $-1\le B<A\le 1$, define the function $K_{A,B}(z)$ by
\begin{equation}\label{p-00010}
K_{A,B}(z)=\begin{cases}
ze^{Az} & \text{ if }~B=0,\\
z(1+Bz)^{\frac{A}{B}-1} & \text{ if }~ B\neq 0.
\end{cases}
\end{equation}
It is easy to show that $K_{A,B}(z)$ belongs to the class $\mathcal{S}\mbox{*}(A,B)$. The function $K_{A,B}(z)$ plays the role of extremal function for many extremal problems in the class $\mathcal{S}\mbox{*}(A,B)$.\\

\noindent The pre-Schwarzian and Schwarzian derivatives for a locally univalent function $f$ are defined by
$$T_f(z)=\frac{f''(z)}{f'(z)} \quad \text{and} \quad S_f(z)=(T_f(z))^{'}-\frac{1}{2}(T_f(z))^2$$
respectively. Also, the pre-Schwarzian and Schwarzian norms of $f$ are defined by
$$||T_f||=\sup\limits_{z\in\mathbb{D}}(1-|z|^2)|T_f(z)| \quad \text{and}\quad ||S_f||=\sup\limits_{z\in\mathbb{D}}(1-|z|^2)^2|S_f(z)|$$
respectively. These norms have significant meanings in the theory of Teichm\"{u}ller space. For a univalent function $f$, it is well known that $||T_f||\leq 6$ and $||S_f||\leq 6$ (see \cite{Nehari-1949}) and these estimates are best possible. On the other hand, for a locally univalent function $f$ in $\mathcal{A}$, it is also known that if $||T_f||\leq 1$ (see \cite{Becker-1972}, \cite{Becker-Pommerenke-1984}) or $||S_f||\leq 2$ (see \cite{Nehari-1949}), then the function $f$ is univalent in $\mathbb{D}$. In 1976, Yamashita \cite{Yamashita-1976} proved that $||T_f ||$ is finite if and only if $f$ is uniformly locally univalent in $\mathbb{D}$.\\

Estimates of the pre-Schwarzian and Schwarzian norms for certain class of analytic and univalent functions has been studied by many researchers. More focus has been given to pre-Schwarzian norm compared to Schwarzian norm due to computational difficulty. We would like to mention some of the promising works in this direction. For $0<\alpha\le 1$, a function $f\in\mathcal{A}$ is called strongly starlike function of order $\alpha$ if $\left|\arg\{zf'(z)/f(z)\}\right|<\pi\alpha/2~ \text{for}~z\in\mathbb{D}$. In 1998, Sugawa \cite{Sugawa-1998} studied the pre-Schwarzian  norm for strongly starlike functions of order $\alpha$, $0<\alpha\le 1$, and proved that $||T_f||\leq M(\alpha)+2\alpha,$
where
$$M(\alpha) =\frac{4\alpha c(\alpha)}{(1-\alpha)c(\alpha)^2+1+\alpha}$$
and $c(\alpha)$ is the unique solution of the following equation in $(1,\infty)$:
$$(1-\alpha)x^{\alpha +2}+(1+\alpha)x^{\alpha}-x^2-1=0$$
and the estimate is sharp. In 1999, Yamashita \cite{Yamashita-1999} studied the pre-Schwarzian norm for the classes $\mathcal{S}\mbox{*}(\alpha)$ and $\mathcal{C}(\alpha)$, where $0\leq \alpha<1$ and proved the sharp estimates
$||T_f||\leq 6-4\alpha$ for $f\in\mathcal{S}\mbox{*}(\alpha)$ and $||T_f||\leq 4(1-\alpha)  ~~\text{for}~ f\in\mathcal{C}(\alpha)$
and the estimates are sharp. A function $f\in\mathcal{A}$ is said to be Janowski convex function if
$1+zf''(z)/f'(z)\prec (1+Az)/(1+Bz)$,
where $-1\leq B<A\leq 1.$ The set of all Janowski convex functions is denoted by $\mathcal{C}(A,B)$. In 2006, Kim and Sugawa \cite{Kim-Sugawa-2006} studied the pre-Schwarzian norm for the class $\mathcal{C}(A,B)$ and proved that
$||T_f||\leq 2(A-B)/(1+\sqrt{1-B^2})$
and the estimate is sharp. A function $f\in \mathcal{A}$ is said to be $\alpha$-spirallike function if ${\rm Re\,}\{e^{-i\alpha}zf'(z)/f(z)\}>0$, where $\alpha\in(-\pi/2,\pi/2)$. In 2000, Okuyama \cite{Okuyama-2000} obtained the sharp estimate of the pre-Schwarzian norm for $\alpha$-spirallike functions. In 2014, Aghalary and Orouji \cite{Aghalary-Orouji-2014} obtained the sharp estimate of the pre-Schwarzian norm for $\alpha$-spirallike functions of order $\rho $, where $-\frac{\pi}{2}<\alpha<\frac{\pi}{2}$ and $0\leq \rho<1$.\\

In the present article, we consider the class $\mathcal{S}\mbox{*}(A,B)$ and find the sharp estimate of the pre-Schwarzian norm for functions in this class. The following result regarding Schwarz functions will be required to prove our main result.

\begin{lem}\label{p-lemma0001}\cite[Lemma 2.2]{Aghalary-Orouji-2014}
Suppose that $\omega:\mathbb{D}\rightarrow \mathbb{D}$ is an analytic function with $\omega(0)=0$. If $-1<A\le 1$, then
$$M(z):=\frac{(|z|^2-|\omega(z)|^2)+|\omega(z)|(1-|z|^2)|2+A\omega(z)|}{|z|(1-|\omega(z)|^2)|2+A\omega(z)|}\leq 1, \quad z\in \mathbb{D}.$$
\end{lem}

\section{Main results}
Before we state our main result, we prove the following lemma which has an independent interest.

\begin{lem}\label{p-lemma0005}
For $-1\leq B<A\leq 1$, let
\begin{equation}
\phi_{A,B}(z)=\frac{2+Az}{(1+Az)(1+Bz)}.
\end{equation}
Then $\phi_{A,B}(z)$ satisfies the following sharp inequalities.

\begin{enumerate}[(i)]
\item If $AB\geq 0$ then for every $z\in \mathbb{D}$, we have
\begin{equation}\label{p-0010a}
|\phi_{A,B}(z)|\leq
\begin{cases}
\phi_{A,B}(-|z|) & \text{ for }~ 0\leq B<A\leq 1\\[2mm]
\phi_{A,B}(|z|) & \text{ for }~ -1\leq B<A\leq 0.
\end{cases}
\end{equation}

\item If $AB<0$ then for every $z\in \mathbb{D}$, we have
\begin{equation}\label{p-0010b}
|\phi_{A,B}(z)|\leq
\begin{cases}
\phi_{A,B}(|z|) & \text{ for }~ |z|^2\le \frac{A+2B}{A^2B}\\[2mm]
\phi_{A,B}(-|z|) & \text{ for }~ |z|^2\ge \frac{A+2B}{A^2B}.
\end{cases}
\end{equation}
\end{enumerate}
\end{lem}

\begin{proof}
Let $z=x+iy\in\mathbb{D}$ and $|z|=r$, $0<r<1$. Then
\begin{align}\label{p-0015}
|\phi_{A,B}(z)|^2
&\le \max\limits_{|z|=r}|\phi_{A,B}(z)|^2 = \max\limits_{|z|=r}\left|\frac{2+Az}{(1+Az)(1+Bz)}\right|^2\\
&= \max\limits_{|z|=r}\frac{4+A^2|z|^2+4A{\rm Re\,}(z)}{(1+A^2|z|^2+2A{\rm Re\,}(z))(1+B^2|z|^2+2B{\rm Re\,}(z))}\nonumber\\
& =\max\limits_{-r\leq x \leq r}\psi(x)\nonumber\\
& \leq \left(\max\limits_{-r\leq x \leq r}\psi_1(x)\right) \left(\max\limits_{-r\leq x \leq r}\psi_2(x)\right)\nonumber,
\end{align}
where
\begin{equation} \label{p-0020}
\psi(x)=\frac{4+A^2r^2+4Ax}{(1+A^2r^2+2Ax)(1+B^2r^2+2Bx)},
\end{equation}
\begin{equation*}\label{p-0023}
\psi_1(x)=\frac{4+4Ax+A^2r^2}{1+2Ax+A^2r^2}\quad\text{and}\quad\psi_2(x)=\frac{1}{1+2Bx+B^2r^2}.
\end{equation*}
A simple calculation gives,
\begin{equation*}\label{p-0025}
\psi'(x)=\frac{h(x)}{(1+A^2r^2+2Ax)^2(1+B^2r^2+2Bx)^2},
\end{equation*}
where
\begin{align}\label{p-0030}
h(x)=&-16A^2Bx^2-8AB(4+A^2r^2)x\\
&-(4A+8B-2A^3r^2+10A^2Br^2+4AB^2r^2+2A^4Br^4-2A^3B^2r^4).\nonumber
\end{align}
Also,
\begin{equation}\label{p-0030a}
\psi_1'(x)=\frac{2A(-2+A^2r^2)}{(1+A^2r^2+2Ax)^2} \quad\text{and}\quad \psi_2'(x)=-\frac{2B}{(1+B^2r^2+2Bx)^2}.
\end{equation}
$(i)$ Let $0\leq B<A\leq 1$. From \eqref{p-0030a}, it is easy to show that $\psi_1$ and $\psi_2$ are decreasing functions on $[-r,r]$. Thus, $\psi_1$ and $\psi_2$ have maximum at $x=-r$. Consequently, the first inequality of \eqref{p-0010a} follows from \eqref{p-0015}.\\

On the other hand, if $-1\leq B<A\leq 0$, from \eqref{p-0030a}, it is easy to see that $\psi_1$ and $\psi_2$ are increasing functions on $[-r,r]$. Thus, $\psi_1$ and $\psi_2$ have maximum at $x=r$. Consequently, the second inequality of \eqref{p-0010a} follows from \eqref{p-0015}.
\\

$(ii)$ Let $AB<0$. Then from \eqref{p-0030}, we have
$$h'(x)=-32AB(1+Ax)-8A^3Br^2>0\quad\text{for all}\quad x\in[-r,r].$$
Thus, the function $h(x)$ is strictly increasing function on $[-r,r]$. If $h$ has a zero in $(-r,r)$, then it is unique, say $x_0$. At $x_0$, it is easy to see that
$$\psi''(x_0)=\frac{h'(x_0)}{(1+A^2r^2+2Ax_0)^2(1+B^2r^2+2Bx_0)^2}>0,$$
which shows that the function $\psi$ has a minimum at $x_0$. That is, the function $\psi$ has maximum at the endpoint $x=r,~-r$. From \eqref{p-0020}, it is not very difficult to verify that
\begin{align*}
\psi(-r)\ge \psi(r) &\iff \frac{2-Ar}{(1-Ar)(1-Br)}\ge \frac{2+Ar}{(1+Ar)(1+Br)}\\
&\iff A+2B-A^2Br^2\ge 0.
\end{align*}
Thus, the maximum of $\psi(x)$ attains at $x=-r$ if $A+2B-A^2Br^2\ge 0$ and at $x=r$ if $A+2B-A^2Br^2\le0$. That is,
\begin{align*}
\max\limits_{-r\leq x \leq r}\psi(x)=
\begin{cases}
\psi(r) &\text{for}~ r^2\le \frac{A+2B}{A^2B}\\[2mm]
\psi(-r) &\text{for}~ r^2\ge \frac{A+2B}{A^2B}.\\[2mm]
\end{cases}
\end{align*}
The desired inequality \eqref{p-0010b} now follows from \eqref{p-0015}.
\end{proof}

Before we state our main results, we introduce some notations which we will use throughout this section. Let
\begin{equation}\label{p-0040a}
\gamma_1(x)=\frac{(1-x^2)(2-Ax)}{(1-Ax)(1-Bx)},\quad
\gamma_2(x)=\frac{(1-x^2)(2+Ax)}{(1+Ax)(1+Bx)},
\end{equation}
\begin{equation}\label{p-0040b}
h_1(x)=A^2Bx^4-2A(A+B)x^3+[5A+(2+A^2)B]x^2-4(1+AB)x+(A+2B),
\end{equation}
\begin{equation}\label{p-0040c}
h_2(x)=A^2Bx^4+2A(A+B)x^3+[5A+(2+A^2)B]x^2+4(1+AB)x+(A+2B).
\end{equation}

\begin{thm}\label{theorem-001}
For $-1\leq B<A\leq 1$, let $f\in\mathcal{S}\mbox{*}(A,B)$ be of the form (\ref{p-00001}). Then we have the sharp inequality $||T_f||\leq ||T_{K_{A,B}}||$, where $K_{A,B}$ is given by \eqref{p-00010}.
\begin{enumerate}[(i)]
\item If $AB\ge 0$, then
$$||T_{K_{A,B}}||=\begin{cases}
(A-B)\gamma_1(\alpha_1), & \text{ for }~ 0\le B<A<1,\\[2mm]
(A-B)\gamma_1(\alpha_1), & \text{ for }~ A=1,~B<\frac{1}{3},\\[2mm]
2, & \text{ for }~ A=1,~B\ge \frac{1}{3},\\[2mm]
(A-B)\gamma_2(\alpha_2), & \text{ for }~ -1< B<A\le 0,\\[2mm]
2(2+A), & \text{for}~B=-1, -1<A\le 0,
\end{cases}$$
where $\gamma_1$ and $\gamma_2$ are given by \eqref{p-0040a} and $\alpha_1$, $\alpha_2$ are the unique roots in $(0,1)$ of the equations $h_1(x)=0$, $h_2(x)=0$ respectively and $h_1$, $h_2$ are given by \eqref{p-0040b}, \eqref{p-0040c} respectively.

\item
If $AB<0$ with $\beta=\sqrt{(A+2B)/A^2B}$, then
$$||T_{K_{A,B}}||=\begin{cases}
(A-B)\gamma_1(\alpha_1), & \text{ for }~ A+2B> 0,\\[2mm]
2(A-B), & \text{ for } ~A+2B= 0,\\[2mm]
(A-B)\max\{\gamma_2(\alpha_2),\gamma_1(\beta)\}, & \text{ for }~ A+2B<0,~\beta<1,\\[2mm]
(A-B)\gamma_2(\alpha_2), & \text{ for }~ A+2B<0,~\beta\ge 1,~B\ne -1,\\[2mm]
2(2+A), & \text{for}~A>0,~B=-1,
\end{cases}$$
where $\gamma_1$ and $\gamma_2$ are given by \eqref{p-0040a} and $\alpha_1$, $\alpha_2$ are the unique roots in $(0,1)$ of the equations $h_1(x)=0$, $h_2(x)=0$ respectively and $h_1$, $h_2$ are given by \eqref{p-0040b}, \eqref{p-0040c} respectively.
\end{enumerate}
\end{thm}

\begin{proof}
Let $f\in \mathcal{S}\mbox{*}(A,B)$. Then from \eqref{p-00005}, we have
\begin{equation*}\label{p-000001}
\frac{zf'(z)}{f(z)}\prec \frac{1+Az}{1+Bz}.
\end{equation*}
Thus, there exists an analytic function $\omega:\mathbb{D}\rightarrow\mathbb{D}$ with $\omega(0)=0$ such that $$\frac{zf'(z)}{f(z)}= \frac{1+A\omega(z)}{1+B\omega(z)}.$$
Taking logarithmic derivatives on both sides and then a simple computation gives
\begin{equation*}\label{p-000005}
T_f(z)=\frac{f''(z)}{f'(z)}=(A-B)\frac{\omega'-\frac{\omega}{z}+\frac{\omega}{z}(2+A\omega)}{(1+A\omega)(1+B\omega)},
\end{equation*}
and so
\begin{equation}\label{p-000010}
\displaystyle (1-|z|^2)|T_f(z)|= (A-B)\frac{(1-|z|^2)|\omega
'-\frac{\omega}{z}+\frac{\omega}{z}(2+A\omega)|}{|1+A\omega||1+B\omega)|}.
\end{equation}
Setting $\omega=id_{\mathbb{D}}$ (the identity function in $\mathbb{D}$), we also have
\begin{equation*}\label{p-000015}
T_{K_{A,B}}(z)=\frac{K''_{A,B}(z)}{K'_{A,B}(z)}=(A-B)\frac{2+Az}{(1+Az)(1+Bz)},
\end{equation*}
and so
\begin{equation}\label{p-000020}
\displaystyle (1-|z|^2)|T_{K_{A,B}}(z)|=(A-B)\frac{(1-|z|^2)|2+Az|}{|1+Az||1+Bz|}.
\end{equation}
By Schwarz-Pick lemma for the function $\frac{\omega(z)}{z}$, we have
\begin{equation*}\label{p-000025}
\displaystyle (1-|z|^2)\left|\omega'-\frac{\omega}{z}\right|\leq \frac{|z|^2-|\omega|^2}{|z|}.
\end{equation*}
From \eqref{p-000010}, we have
\begin{align*}
(1-|z|^2)|T_f(z)|
&\leq (A-B)\left[\frac{(1-|z|^2)|\omega'-\frac{\omega}{z}|}{|1+A\omega||1+B\omega|}+\frac{|\omega|}{|z|}\frac{(1-|z|^2)|2+A\omega|}{|1+A\omega||1+B\omega|}\right]\\
&\leq (A-B)\frac{(1-|\omega|^2)|2+A\omega|}{|1+A\omega||1+B\omega|}~\frac{(|z|^2-|\omega|^2)+|\omega|(1-|z|^2)|2+A\omega|}{|z|(1-|\omega|^2)|2+A\omega|}\\
&\leq (A-B)\frac{(1-|\omega|^2)|2+A\omega|}{|1+A\omega||1+B\omega|}~M(z),
\end{align*}
where $M(z)$ is defined as in Lemma \ref{p-lemma0001}.
By Lemma \ref{p-lemma0001}, we have
\begin{equation*}\label{p-000030}
(1-|z|^2)|T_f(z)|\leq (1-|\omega|^2)|T_{K_{A,B}}(\omega)|\leq ||T_{K_{A,B}}||,
\end{equation*}
which shows that $||T_f||\leq ||T_{K_{A,B}}||$.\\

Now, we will concentrate on estimating the pre-Schwarzian norm of the function $K_{A,B}(z)$. It follows from \eqref{p-000020} that
\begin{equation}\label{p-000035}
||T_{K_{A,B}}||=\sup\limits_{z\in\mathbb{D}}(1-|z|^2)|T_{K_{A,B}}(z)| =(A-B)\sup\limits_{z\in \mathbb{D}}(1-|z|^2)|\phi_{A,B}(z)|,
\end{equation}
where $\phi_{A,B}(z)$ is defined as in Lemma \ref{p-lemma0005}.\\

We note that if $AB\ge 0$, then either $0\le B<A\le 1$ or $-1\le B<A\le 0$. On the other hand if $AB<0$, then $-1\le B<0<A\le 1$. Now we consider the following cases.\\

\noindent\textbf{Case-1: } Let $AB\ge 0$ with $0\le B<A\le 1$. Then by Lemma \ref{p-lemma0005}, we obtain
\begin{equation}\label{p-000040}
\max_{|z|=r}\left|\phi_{A,B}(z)\right|\leq \phi_{A,B}(-r),
\end{equation}
which is sharp and the equality occurs when $z=-r$, $0<r<1.$ Using \eqref{p-000040} in \eqref{p-000035}, yields
\begin{align}\label{p-000040aa}
||T_{K_{A,B}}||&=\sup\limits_{z\in \mathbb{D}}(1-|z|^2)|T_{K_{A,B}}(z)| =(A-B)\sup\limits_{z\in \mathbb{D}}(1-|z|^2)|\phi_{A,B}(z)|\\
&=(A-B)\sup\limits_{0<x<1}\gamma_1(x),\nonumber
\end{align}
where $\gamma_1$ is given by \eqref{p-0040a}. A simple computation gives
\begin{equation}\label{p-000040a}
\gamma'_1(x)=\frac{h_1(x)}{(1-Ax)^2(1-Bx)^2},
\end{equation}
where $h_1$ is given by \eqref{p-0040b}. Clearly, $h_1(0)=A+2B>0$ and $h_1(1)=2(A-1)(1-B)(2-A)\le0$. Moreover,
\begin{equation}\label{p-000042a}
\begin{aligned}
\begin{cases}
h'_1(x)&=-4(1+AB)+2[5A+(2+A^2)B]x-6A(A+B)x^2+4A^2Bx^3,\\
h''_1(x)&=2[5A+(2+A^2)B]-12A(A+B)x+12A^2Bx^2,\\
h'''_1(x)&=-12A^2(1-Bx)-12AB(1-Ax)<0.
\end{cases}
\end{aligned}
\end{equation}
Thus, $h''_1(x)$ is a strictly decreasing function on $(0,1)$. Further,
\begin{equation}\label{p-000042b}
\begin{aligned}
\begin{cases}
& h'_1(0)=-4(1+AB)<0,\\
& h'_1(1)=2(1-A)(1-B)(3A-2),\\
&h''_1(0)=2[5A+(2+A^2)B]>0,\\ &h''_1(1)=10A+4B+14A^2B-12A^2-12AB.
\end{cases}
\end{aligned}
\end{equation}
Now, we consider the following subcases.\\

\noindent \textbf{Subcase-1a: } Let $A\le\frac{2}{3}$. Then it is easy to see from \eqref{p-000042b} that
$$h''_1(1)=(2-3A)\left[4A(1-B)+\frac{2}{3}B\left(\frac{4}{3}-A\right)\right]+2A+\frac{20}{9}B>0.$$
Therefore, $h'_1$ is a strictly increasing function on $(0,1)$ with $h'_1(1)\le 0$. Thus, $h_1$ is a strictly decreasing function on $(0,1)$. Consequently, the function $h_1(x)$ has exactly one zero in $(0,1)$.\\

\noindent \textbf{Subcase-1b: } Let $\frac{2}{3}<A<1$. From \eqref{p-000042b}, $h_1'(0)<0$ and $h_1'(1)>0$ and so $h'_1(x)$ has atleast one zero in $(0,1)$. Since $h'_1$ is a continuous polynomial of degree $3$, the number of zeros of $h'_1$ in $(0,1)$ is either $1$ or $3$. If $h_1'(x)$ has $3$ zeros in $(0,1)$, then by Rolle's theorem $h_1''(x)$ must have $2$ zeros in $(0,1)$, which is not possible as $h''_1$ is a strictly decreasing function on $(0,1)$. Thus, $h'_1$ has unique zero, say $\alpha_0$, in $(0,1)$. Therefore, $h_1$ is a decreasing function on $(0,\alpha_0)$ and increasing function on $(\alpha_0,1)$. Since $h_1(0)>0$ and $h_1(1)<0$, the function $h_1$ has exactly one zero in $(0,1)$.\\

In the above two subcases, the function $h_1(x)$ has unique zero, say $\alpha_1$, in $(0,1)$. Since $\gamma'_1(0)=h_1(0)>0$ and $\gamma'_1(1)<0$, the function $\gamma_1(x)$ is increasing in a neighborhood of $0$ and decreasing in a neighborhood of $1$. Therefore, $\gamma_1$ can not attain its maximum at $x=0,1$ and so $\gamma_1$ has maximum at $\alpha_1$. From \eqref{p-000040aa}, we get the desired result in these subcases.\\

\noindent\textbf{Subcase-1c: } Let $A=1$ and $B<\frac{1}{3}$. Then the function $h_1(x)$ given by \eqref{p-0040b} can be written as
\begin{equation}\label{p-000045}
h_1(x)=(1-x)^2 q(x), \quad\text{where}~ q(x)=1+2B-2x+Bx^2.
\end{equation}
Clearly, $q'(x)=-2(1-Bx)<0$ for all $x\in(0,1)$ and so $q$ is a strictly decreasing function on $(0,1)$. Moreover, $q(0)=1+2B>0$, $q(1)=3B-1< 0$. Thus, the function $q$ has unique zero in $(0,1)$ and so the function $h_1$ has unique zero in $(0,1)$, say $\alpha_1$. Since $\gamma'_1(0)>0$ and $\gamma'_1(1)=q(1)/(1-B)^2<0$, the function $\gamma_1(x)$ is increasing in a neighborhood of $0$ and decreasing in a neighborhood of $1$. Thus, $\gamma_1$ can not attain its maximum at $x=0,1$ and so $\gamma_1$ has maximum at $\alpha_1$. From \eqref{p-000040aa}, we get the desired result in this subcase.\\

\noindent\textbf{Subcase-1d: } Let $A=1$ and $B\ge\frac{1}{3}$. Then from \eqref{p-000045}, we have $q(0)=1+2B>0$, $q(1)=-1+3B\geq 0$. Since $q$ is a strictly decreasing function on $(0,1)$, it follows that $h_1(x)>0$ for all $x\in(0,1)$. Therefore, $\gamma_1$ is a strictly increasing function on $(0,1)$. Thus, $\displaystyle \sup_{0<x<1}\gamma_1(x)=\gamma_1(1)=2/(1-B)$ and so from \eqref{p-000040aa}, we get $||T_{K_{1,B}}||=2.$\\

\noindent\textbf{Case-2: } Let $AB\ge 0$ with $-1\le B<A\le 0$. By Lemma \ref{p-lemma0005}, we have
\begin{equation}\label{p-000050}
\max_{|z|=r}\left|\phi_{A,B}(z)\right|\leq \phi_{A,B}(r),
\end{equation}
which is sharp and the equality occurs at $z=r$, $0<r<1.$ Using \eqref{p-000050} in \eqref{p-000035}, yields
\begin{equation}\label{p-000053}
\begin{aligned}
||T_{K_{A,B}}||&=\sup\limits_{z\in \mathbb{D}}(1-|z|^2)|T_{K_{A,B}}(z)| =(A-B)\sup\limits_{z\in \mathbb{D}}(1-|z|^2)|\phi_{A,B}(z)|\\
&=(A-B)\sup\limits_{0<x<1}\gamma_2(x),
\end{aligned}
\end{equation}
where $\gamma_2$ is given by \eqref{p-0040a}.\\

\noindent \textbf{Subcase-2a: } Let $B\neq -1$. A simple computation gives
$$\gamma'_2(x)=-\frac{h_2(x)}{(1+Ax)^2(1+Bx)^2},$$
where $h_2(x)$ is given by \eqref{p-0040c}. Clearly, $h_2(0)=A+2B<0$, $h_2(1)=2(1+A)(1+B)(2+A)>0$. Moreover,
\begin{align*}
h'_2(x)&=4(1+AB)+2[5A+(2+A^2)B]x+6A(A+B)x^2+4A^2Bx^3,\\
h''_2(x)&=2[5A+(2+A^2)B]+12A(A+B)x+12A^2Bx^2.
\end{align*}
The discriminant of $h''_2(x)=0$ is $48A^2[3A(A-B)-AB(1+2AB)-B^2]<0$. Thus, $h''_2(x)=0$ has no real roots. Since $h_2'(x)$ is a polynomial of degree $3$, so $h'_2(x)=0$ has only one real root. Thus, $h_2(x)=0$ have two real roots. Clearly, $h_2(x)=0$ has one real root on each of the intervals $(0,1)$ and $(1,\infty)$. Consequently, $h_2$ has exactly one zero, say $\alpha_2$, in $(0,1)$. Since $\gamma'_2(0)>0$ and $\gamma'_2(1)<0$, the function $\gamma_2(x)$ is increasing in a neighborhood of $0$ and decreasing in a neighborhood of $1$ and so $\gamma_2$ has maximum at $\alpha_2$. From \eqref{p-000053}, we get the desired result in this subcase.\\

\noindent \textbf{Subcase-2b: } Let $B=-1$. Then $\gamma_2$ can be written as $\gamma_2(x)=(1+x)(2+Ax)/(1+Ax)$. A simple calculation shows that $\gamma_2$ is increasing function on $(0,1)$ and therefore, from \eqref{p-000053}, we get $||T_{K_{A,-1}}||=2(2+A)$.\\

\noindent\textbf{Case-3: } Let $AB<0$ with $A+2B> 0$. From the Lemma \ref{p-lemma0005}, we conclude that
\begin{equation}\label{p-000060}
\max_{|z|=r}\left|\phi_{A,B}(z)\right|\leq \phi_{A,B}(-r),
\end{equation}
which is sharp and the equality occurs at $z=-r$, $0<r<1.$ Using \eqref{p-000060} in \eqref{p-000035}, we have
\begin{align}\label{p-000061a}
||T_{K_{A,B}}||&=\sup\limits_{z\in \mathbb{D}}(1-|z|^2)|T_{K_{A,B}}(z)| =(A-B)\sup\limits_{z\in \mathbb{D}}(1-|z|^2)|\phi_{A,B}(z)|\\
&=(A-B)\sup\limits_{0<x<1}\gamma_1(x),\nonumber
\end{align}
where $\gamma_1(x)$ is given by \eqref{p-0040a}. Then, $\gamma_1'(x)$ and $ h_1(x)$ are given by \eqref{p-000040a} and \eqref{p-0040b} respectively and $h'_1(x),~h''_1(x),~h'''_1(x)$ are given by \eqref{p-000042a}. Clearly,
\begin{equation}\label{p-000063}
\begin{cases}
& h_1(0)=A+2B>0,\\ &h_1(1)=2(A-1)(1-B)(2-A),\\
& h'_1(0)=-4(1+AB)<0,\\ &h'_1(1)=2(1-A)(1-B)(3A-2),\\
&h''_1(0)=2[A(4+AB)+(A+2B)]>0,\\ &h''_1(1)=10A+4B+14A^2B-12A^2-12AB.
\end{cases}
\end{equation}
From \eqref{p-000042a}, it is easy to show that $h'''_1(x)<0$ for all $x\in(0,1)$. Therefore, $h''_1$ is a strictly decreasing function on $(0,1)$. Now, we consider the following subcases. \\

\noindent \textbf{Subcase-3a: } Let $A\le \frac{2}{3}$. From \eqref{p-000063} it is easy to see that
$$h''_1(1)=(2-3A)\left(-\frac{14AB}{3}+4A\right)+2(A+2B)-\frac{8}{3}AB>0.$$
Therefore, $h'_1$ is a strictly increasing function on $(0,1)$ and $h'_1(1)\le 0$, i.e., $h'_1(x)< 0$ for all $x\in (0,1)$. Thus, $h_1$ is a strictly decreasing function on $(0,1)$. This shows that $h_1$ has exactly one zero in $(0,1)$.\\

\noindent \textbf{Subcase-3b: } Let $\frac{2}{3}<A<1$. Then from \eqref{p-000063}, the function $h'_1$ has atleast one zero in $(0,1)$. Indeed, $h'_1$ has three zeros, one lies in each of the intervals $(-\infty,0)$, $(0,1)$ and $(1,\infty).$ Therefore, $h'_1$ has unique zero, say $\alpha_0$, in $(0,1)$. Thus, $h_1$ is decreasing in $(0,\alpha_0)$ and increasing in $(\alpha_0,1)$. Since $h_1(0)>0$ and $h_1(1)<0$, it follows that $h_1$ has unique zero in $(0,1)$.\\

\noindent\textbf{Subcase-3c: } Let $A=1$ and $-\frac{1}{2}<B\le 0$. From \eqref{p-000045}, we have $q(0)=1+2B>0$, $q(1)=-1+3B< 0$. Since $q$ is a strictly decreasing function on $(0,1)$, it follows that $q(x)$ has unique zero in $(0,1)$ and so $h_1(x)$ has also unique zero in $(0,1)$.\\

In the above three subcases, the function $h_1(x)$ has an unique zero, say $\alpha_1$, in $(0,1)$. Since $\gamma'_1(0)=h_1(0)>0$ and $\gamma'_1(1)<0$, the function $\gamma_1(x)$ is increasing in a neighborhood of $0$ and decreasing in a neighborhood of $1$. Thus, $\gamma_1$ can not attain its maximum at $x=0,1$ and so $\gamma_1$ has maximum at $\alpha_1$. From \eqref{p-000061a}, we get the desired result in these subcases.\\

\noindent\textbf{Case-4: } Let $AB<0$ with $A+2B=0$. Note that $AB<0$ and $A+2B=0$ implies that $-\frac{1}{2}\leq B<0$. Then as in \textbf{Case-3}
\begin{align*}
||T_{K_{A,B}}||=(A-B)\sup\limits_{0<x<1}\gamma_1(x),
\end{align*}
where $\gamma_1(x)$ is given by \eqref{p-0040a}. From \eqref{p-0040b}, $h_1$ can be written as
$h_1(x)=4xp(x)$,
where 
$$p(x)=B^3x^3-B^2x^2+B^3x-2Bx+2B^2-1.$$
\noindent Clearly, $p(0)=2B^2-1<0$ and $p(1)=(B^2-1)(1+2B)\le 0$. Moreover,
\begin{align*}
p'(x)&=3B^3x^2-2B^2x+B^3-2B,\\
p''(x)&=6B^3x-2B^2<0,
\end{align*}
which shows that $p'$ is a strictly decreasing function on $(0,1)$. Thus, $p'(0)=B(B^2-2)>0$ and $p'(1)=B(4B^2-2B-2)\ge 0$. It follows that $p'(x)> 0$ for all $x\in (0,1)$. Therefore, $p$ is a  strictly increasing function on $(0,1)$ with $p(1)\le 0$. This lead us to conclude that $h_1(x)< 0$ for all $x\in (0,1)$ and so $\gamma'_1(x)< 0$ for all $x$ in $(0,1)$. That is, $\gamma_1$ is a strictly decreasing function on $(0,1)$. Therefore, $\gamma_1$ attains maximum at $0$ and consequently, $||T_{K_{A,B}}||=2(A-B)$.\\

\noindent\textbf{Case-5: } Let $AB<0$ with $A+2B<0$ and $\beta=\sqrt{(A+2B)/A^2B}<1$. From the Lemma \ref{p-lemma0005}, we obtain
\begin{equation}\label{p-000065}
\max_{|z|=r}\left|\phi_{A,B}(z)\right|\leq \phi_{A,B}(r), ~\text{if}~0<r\leq \beta,
\end{equation}
and
\begin{equation}\label{p-000070}
\max_{|z|=r}\left|\phi_{A,B}(z)\right|\leq \phi_{A,B}(-r),~\text{if}~\beta\leq r<1.
\end{equation}
Both the inequalities are sharp and equality occurs in \eqref{p-000065} at $z=r$ and equality occurs in \eqref{p-000070} at $z=-r$, $0<r<1$. Therefore,
\begin{align}\label{p-000073}
||T_{K_{A,B}}||&=\sup\limits_{z\in \mathbb{D}}(1-|z|^2)|T_{K_{A,B}}(z)| =(A-B)\sup\limits_{z\in \mathbb{D}}(1-|z|^2)|\phi_{A,B}(z)|\\
&=(A-B)\max\left\{\sup\limits_{0<x\leq \beta}\gamma_2(x),\sup\limits_{\beta\leq x<1}\gamma_1(x)\right\},\nonumber
\end{align}
where $\gamma_1$ and $\gamma_2$ are given by \eqref{p-0040a}. We now prove that $\gamma_2$ has only one zero in $(0,\beta)$ and $\gamma_1$ has no zero in $(\beta,1)$ by considering two different subcases:\\

\noindent \textbf{Subcase-5a: } Let $0<x<\beta<1$. In this subcase
$$\gamma'_2(x)=-\frac{h_2(x)}{(1+Ax)^2(1+Bx)^2},$$
where $h_2(x)$ is given by \eqref{p-0040c}. Since $\beta<1$, it is easy to show that $A+B>B(A^2-1)\ge 0$.
Therefore, $h_2(0)=A+2B<0$ and $h_2(\beta)={\beta}^2[2A(1+B)+4(A+B)]+\beta[4(1+AB)+2A(A+B){\beta}^2]>0$. This shows that $h_2$ has atleast one zero in $(0,\beta)$. Since $-1\le B<0<A\le 1$, it follows that
\begin{align}
\label{p-000075}
h'_2(x)&=(4+10Ax+6A^2x^2)+(4AB+4Bx+2A^2Bx+6ABx^2+4A^2Bx^3)\\&
\ge (4+10Ax+6A^2x^2)-(4A+4x+2A^2x+6Ax^2+4A^2x^3)\nonumber \\
&=4(1-x)(1-A)+2Ax(3-A)(1-x)+4A^2x^2(1-x)\nonumber\\
& >0\nonumber
\end{align}
for all $x\in (0,1)$. Thus, $h_2$ is a strictly increasing function on $(0,1)$. Consequently, $h_2$ has exactly one zero, say $\alpha_2$, in $(0,\beta)$. Since $\gamma'_2(0)=-h_2(0)>0$ and $\gamma_2'(\beta)<0$,  the function $\gamma_2(x)$ is increasing in a neighborhood of $0$ and decreasing in a neighborhood of $\beta$. Thus, $\gamma_2$ has maximum at $\alpha_2$.\\

\noindent \textbf{Subcase-5b: } Let $\beta\le x<1$. In this subcase
$$\gamma'_1(x)=\frac{h_1(x)}{(1-Ax)^2(1-Bx)^2},$$
where $h_1(x)$ is given by \eqref{p-0040b}. By division algorithm, $h_1(x)$ can be written as
\begin{equation}\label{p-000080}
h_1(x)=\frac{1}{2}(A+2B)Q(x)-\frac{x}{2}R(x),
\end{equation}
where
\begin{align*}
Q(x)&=A^2x^4-2Ax^3+2x^2+A^2x^2-4Ax+2,\\ R(x)&=A^3x^3+2A^2x^2-8Ax+A^3x-4A^2+8.
\end{align*}
We claim that $Q(x)$ and $R(x)$ are positive in $(0,1)$. First we prove that $R(x)$ is positive in $(0,1)$. Clearly, $R(0)=8-4A^2>0$, $R(1)=(1-A)(8-2A^2)\ge 0$ and
\begin{align*}
R'(x)&=3A^3x^2+4A^2x-8A+A^3\\&=3A(A^2x^2-1)+4A(Ax-1)+A(A^2-1)<0.
\end{align*}
Thus, $R(x)$ is a strictly decreasing function on $(0,1)$ and so $R(x)$ is positive in $(0,1)$.\\

Now, we prove that $Q(x)$ is positive on $(0,1)$. We take $Q(x)=Q(x,A)=A^2x^4-2Ax^3+2x^2+A^2x^2-4Ax+2$. It is easy to prove that $Q(x,A)$ is a decreasing function of $A$ and $Q(x,1)=(-1+x)^2(2+x^2)>0$ for all $x\in (0,1)$. Thus, $Q$ is positive for all $A$ and $x\in (0,1)$. \\

Consequently, from \eqref{p-000080}, we conclude that $h_1(x)<0$ for all $x$ in $(0,1)$. Therefore, $h_1$ has no zero in $(0,1)$. Therefore, $\gamma_1$ has a maximum at the endpoint. Moreover, $\gamma_1(\beta)>\gamma_1(1)=0$.\\

\noindent Combining the \textbf{Subcase-5a} and \textbf{Subcase-5b}, we get the desired result from \eqref{p-000073}.\\

\noindent \textbf{Case-6: } Let $AB<0$ with $A+2B<0$ and $\beta=\sqrt{(A+2B)/A^2B}\ge 1$. From the Lemma \ref{p-lemma0005}, we obtain
\begin{equation}\label{p-000085}
\max_{|z|=r}\left|\phi_{A,B}(z)\right|\leq \phi_{A,B}(r),
\end{equation}
which is sharp and the equality occurs at $z=r$, $0<r<1$. Using \eqref{p-000085} in \eqref{p-000035}, yields
\begin{align}\label{p-000090}
||T_{K_{A,B}}||&=\sup\limits_{z\in \mathbb{D}}(1-|z|^2)|T_{K_{A,B}}(z)| =(A-B)\sup\limits_{z\in \mathbb{D}}(1-|z|^2)|\phi_{A,B}(z)|\\
&=(A-B)\sup\limits_{0<x<1}\gamma_2(x),\nonumber
\end{align}

where $\gamma_2$ is given by \eqref{p-0040a}. \\

\noindent\textbf{Subcase-6a: } Let  $B=-1$. Then $\gamma_2$ can be written as $\gamma_2(x)=(1+x)(2+Ax)/(1+Ax)$. A simple calculation shows that $\gamma_2$ is a increasing function on $(0,1)$ and therefore, $||T_{K_{A,B}}||=2(2+A)$.\\

\noindent\textbf{Subcase-6b: } Let $B\ne -1$. Therefore,
$$\gamma'_2(x)=-\frac{h_2(x)}{(1+Ax)^2(1+Bx)^2},$$
where $h_2(x)$ is given by \eqref{p-0040c}. In particular, $h_2(0)=A+2B<0$ and $h_2(1)=2(1+A)(1+B)(2+A)>0$. From \eqref{p-000075}, it is easy to show that $h'_2(x)>0$ for all $x$ in $(0,1)$. Therefore, $h_2$ has unique zero, say $\alpha_2$, in $(0,1)$. Since $\gamma'_2(0)=-h_2(0)>0$ and $\gamma'_2(1)<0$, the function $\gamma_2(x)$ is increasing in a neighborhood of $0$ and decreasing in a neighborhood of $1$. Thus, $\gamma_2$ has maximum at $\alpha_2$. Thus, the desired result follows from \eqref{p-000090}.
\end{proof}

For particular values of $A$ and $B$ with $-1\le B<A\le 1$, we get the following sharp estimate for the pre-Schwarzian norm for several subclasses of $\mathcal{S}$.\\

\begin{cor}\cite[Theorem 2]{Aghalary-Orouji-2014}
Let $f\in \mathcal{S}\mbox{*}(\alpha)$, $0\le\alpha<1$ be of the form \eqref{p-00001}. Then the sharp inequality  $||T_f||\le 6-4\alpha$, holds and equality occurs for the function $k_{\alpha}(z)=z/(1-z)^{2(1-\alpha)}$.
\end{cor}

\begin{cor}
Let $f\in \mathcal{S}\mbox{*}$ be of the form \eqref{p-00001}. Then the sharp inequality $||T_f||\le 6$, holds and equality occurs for the Koebe function $k(z)=z/(1-z)^2$.
\end{cor}

\noindent Taking $A=1$ and $B=0$, we get the following result.

\begin{cor}
Let $f\in \mathcal{S}\mbox{*}(1,0)$ be of the form \eqref{p-00001}. Then the sharp inequality $||T_f||\le 9/4$, holds and equality occurs for the function $f(z)=ze^z$.
\end{cor}

\noindent Taking $A=\alpha$ and $B=-\alpha$, $0<\alpha\le 1$, we get the following result.

\begin{cor}
Let $f\in\mathcal{S}\mbox{*}(\alpha,-\alpha)$, $0<\alpha\le 1$ be of the form \eqref{p-00001}. Then the sharp inequality 
$$||T_f||\le\begin{cases}
2\alpha(1-x_0^2)(2+\alpha x_0)/(1-\alpha^2x_0^2), & \text{ for }~ 0<\alpha<1,\\[2mm]
6, & \text{ for }~\alpha=1
\end{cases}$$
holds, where $x_0$ is an unique root in $(0,1)$ of the equation: 
$$\alpha^3x^4+(\alpha^3-3\alpha)x^2+(4\alpha^2-4)x+\alpha=0.$$ Moreover, the equality occurs for the function $f(z)=z/(1-\alpha z)^2$.
\end{cor}


\begin{thebibliography}{99}


\bibitem{Aghalary-Orouji-2014}
{\sc R. Aghalary} and {\sc Z. Orouji}, Norm Estimates of the Pre-Schwarzian Derivatives for $\alpha $-Spiral-Like Functions of Order $\rho $, {\it Complex Anal. Oper. Theory} {\bf 8(4)} (2014), 791--801.


\bibitem{Becker-1972}
{\sc J. Becker}, L\"ownersche Differentialgleichung und quasikonform fortsetzbare schlichte Funktionen, {\it J. Reine Angew. Math}, {\bf255} (1972), 23--43.


\bibitem{Becker-Pommerenke-1984}
J. Becker and  Ch. Pommerenke, Schlichtheitskriterien und Jordangebiete, {\it J. Reine Angew. Math.} {\bf 354} (1984), 74--94.


\bibitem{Duren-1983}
{\sc P. L. Duren}, {\it Univalent functions} (Grundlehren der
mathematischen Wissenschaften 259, New York, Berlin, Heidelberg, Tokyo), Springer-Verlag, 1983.



\bibitem{Goodman-book-1983}
{\sc A. W. Goodman}, Univalent Functions, Vols. I and II. Mariner Publishing Co. Tampa, Florida, 1983.



\bibitem{Janowski-1973}
{\sc W. Janowski}, Some extremal problems for certain families of analytic functions, {\it Ann. Polon. Math.} {\bf 28} (1973), 297--326.



\bibitem{Kim-Sugawa-2006}
{\sc Y. C. Kim} and {\sc T. Sugawa}, Norm estimates of the pre-Schwarzian derivatives for certain classes of univalent functions, {\it Proc. Edinb. Math. Soc. (2)} {\bf 49(1)} (2006), 131--143.


\bibitem{Nehari-1949}
{\sc Z. Nehari}, The Schwarzian derivative and schlicht functions, {\it Bull. Amer. Math. Soc.} {\bf 55(6)} (1949), 545--551.


\bibitem{Okuyama-2000}
{\sc Y. Okuyama}, The Norm estimates of pre-Schwarzian derivatives of spiral-like functions, {\it Complex Var. Theory Appl.} {\bf 42} (2000), 225--239.


\bibitem{Padmanabhan-1968}
{\sc K. S. Padmanabhan}, On certain classes of starlike functions in the unit disk, {\it J. Indian Math. Soc.}, (N.S.) {\bf 32} (1968), 89--103.


\bibitem{Singh-Singh-1974}
{\sc R. Singh} and {\sc V. Singh}, On a class of bounded starlike functions, {\it Indian J. Pure Appl. Math.}, {\bf 5} (1974), 733--754.


\bibitem{Sugawa-1998}
{\sc T. Sugawa}, On the norm of pre-Schwarzian derivatives of strongly starlike functions, {\it Ann. Univ. Mariae Curie-Skłodowska Sect. A} {\bf 52(2)} (1998), 149--157.


\bibitem{Yamashita-1976}
{\sc S. Yamashita}, Almost locally univalent functions, \textit{Monatsh. Math.} \textbf{81} (1976), 235–240.



\bibitem{Yamashita-1999}
{\sc S. Yamashita}, Norm estimates for function starlike or convex of order alpha, {\it Hokkaido Math. J.} {\bf 28(1)} (1999), 217--230.






\end{thebibliography}
\end{document}